\documentclass[a4paper, 11pt]{article}

\usepackage{amsmath}
\usepackage{amsfonts}
\usepackage{graphics}
\usepackage{epsfig}
\usepackage{amssymb}
\usepackage{amsthm}
\usepackage{amscd}
\usepackage[all]{xy}
\usepackage{latexsym}
\usepackage{geometry}
\usepackage{multirow}
\usepackage{geometry}
\usepackage{color}
\usepackage{CJKutf8} 
\usepackage{bm}
\usepackage{mathrsfs} 
\usepackage{titlesec}

\usepackage{ulem}

\newtheorem{thm}{Theorem}[section]

\newtheorem{prop}[thm]{Proposition}

\theoremstyle{definition}

\newtheorem{rem}[thm]{Remark}

\numberwithin{equation}{section}
\numberwithin{figure}{section}
\numberwithin{table}{section}

\def\rchi{{\hbox{\raise1.5pt\hbox{$\chi$}}}}

\newcommand{\bea}{\begin{eqnarray}}
\newcommand{\eea}{\end{eqnarray}}
\newcommand{\be}{\begin{equation}}
\newcommand{\ee}{\end{equation}}

\newcommand{\PIV}{P_{\rm IV}}

\title{\bf  Note on the Singularity Reduction 
of Isomonodromy Systems  Associated with  Garnier Systems}

\author{\Large Kohei Iwaki\footnote{Graduate School of Mathematical Science, 
The University of Tokyo.} \qquad 
Seiya Kato\footnote{College of Arts and Sciences,
The University of Tokyo.} \qquad 
Shotaro Sakurai\footnote{College of Arts and Sciences,
The University of Tokyo.}}

\begin{document}
\large
\setcounter{section}{0}

\maketitle

\begin{abstract}
In this paper, we study the isomonodromy systems associated with 
the Garnier systems of type 9/2 and type 5/2+3/2.  
We show that the both of isomonodromy systems admit 
the singularity reduction (restriction to a movable pole), 
and the resulting linear differential equations 
are isomonodromic with respect to the second variable of the Garnier systems. 
Furthermore, we find two fourth-order nonlinear ordinary differential equations that describe the isomonodromy deformation but lack the Painlev\'e property. 
One of these equations has been already found by Dubrovin--Kapaev in 2014, 
while the other provides a new example.
\end{abstract}

\tableofcontents








\allowdisplaybreaks


\setlength{\parskip}{0.5ex}

\section{Introduction}
\label{sec:intro}

The Painlev\'e equations were originally discovered through the pioneering work 
of Paul Painlev\'e in the early 20th century. 
These six second-order, nonlinear ordinary differential equations 
(known as the $P_{\rm I}$ -- $P_{\rm VI}$) 
are characterized by the so-called Painlev\'e property;
that is, their solutions possess no movable critical singularities. 
These equations are now recognized as a remarkable class of nonlinear ordinary differential equations, 
notable for their intricate mathematical structures and deep connections to areas 
of theoretical physics, including random matrix theory and string theory. 

Among the many fascinating aspects of Painlev\'e equations, 
this paper focuses on 
the isomonodromy deformation and a certain properties of 
their Laurent series solutions (the Painlev\'e test). 
It is known that each Painlev\'e equation describes the isomonodromy deformation 
of a certain linear differential equation with rational coefficients, 
and solutions of the Painlev\'e equation possess the monodromy-Stokes data of 
the linear differential equation as conserved quantities (\cite{JMU81, JM81}). 
On the other hand, the Laurent series expansion of the solutions to 
the Painlev\'e equations at movable poles also possesses a certain nontrivial structure. 
More precisely, when assuming that the solution has a Laurent series expansion 
at a generic reference point, we can show that 
all series coefficients are determined without contradiction, 
they contains as many free parameters as the order of the equation, 
and the Laurent series converges on a punctured neighborhood of the reference point. 
A differential equation having this property is said to pass the (extended) Painlev\'e test
(\cite{ARS80, CM08}; see also \cite{Chi15}). 
These two characteristics both signify the integrability of the Painlev\'e equations.

Previous studies including \cite{IN86, FIKN06, Mas10, DK18} 
have revealed that there is an even more nontrivial structure between 
the isomonodromy systems and the Laurent series solution of the Painlev\'e equations. 
That is, although the solutions to the Painlev\'e equation diverge at movable poles, 
in the associated isomonodromy system (if a suitable gauge transformation is applied if necessary), 
the divergence of the solution at the movable pole can be canceled, 
and it becomes possible to restrict the isomonodromy system to the movable pole. 
Borrowing the terminology used in Dubrovin--Kapaev's paper \cite{DK14}, 
which is particularly relevant to our work, 
we will refer to the limiting procedure that restricts 
the isomonodromy system to the movable pole 
as ``singularity reduction''.
In the aformentioned previous studies, it has been shown that 
the (confluent) Heun equations are obtained through the singularity reduction 
from the isomonodromy systems associated with the Painlev\'e equations.
The free parameters that appear in the Laurent series solution 
during the Painlev\'e test parametrize the Heun equation. 
Furthermore, it is interesting to note that one of these free parameters provides 
the accessory parameters for the Heun equations.

On the other hand,  from the viewpoint of the exact WKB analysis (\cite{Vor83, KT98}), 
the monodromy-Stokes data of a Schr\"odinger-type 
linear differential equations containing a small parameter $\hbar$ 
can be described perturbatively (but non-perturbatively by taking the Borel sum) 
in terms of the period integrals on the algebraic curve that arises 
as the classical limit (c.f., \cite[Section 3]{KT98}). 
Exact WKB analysis can be applied to the Heun differential equations
by introducing $\hbar$ through a rescaling of variables, 
and their classical limit generically give elliptic curves.
Elliptic curves have two independent closed cycles,
and it is consistent with the fact that the Painlev\'e equations are second-order equations.  
In fact, \cite{Iwa19} provides a detailed (albeit heuristic) analysis 
of the relationship between the period integrals of the Weierstrass elliptic curve 
and the Stokes data of the isomonodromy system associated with the first Painlev\'e equation. 
From this perspective, singularity reduction provides a framework for describing 
the Riemann-Hilbert correspondence between the general solutions of the Painlev\'e equations, 
represented as Laurent series, and the monodromy-Stokes data of the isomonodromy systems 
via the period map of elliptic curves.
We do not conduct exact WKB theoretic analysis in this paper, 
but we deal with isomonodromy systems where $\hbar$ has been introduced, 
based on \cite{KT98, Iwa19} and their underlying ideas.

This naturally leads to the expectation that, 
for the fourth-order (or four-dimensional) Painlev\'e equations 
studied in \cite{Kim89, KNS18, Kaw17} etc., 
a similar analysis would yield a genus-2 hyper-elliptic curve as 
the classical limit of the linear differential equation derived via the singularity reduction.
The primary objective of this paper is to explore these questions 
for certain fourth-order Painlev\'e equations, 
which arise as a restriction of the degenerate Garnier systems.

To the best knowledge of the authors, 
the first step for the study was given in the work of Dubrovin--Kapaev \cite{DK14}, 
where they study a fourth-order Painlev\'e equation obtained as 
the restriction of the degenerate Garnier systems of type 9/2. 
While it was shown by Shimomura \cite{Shimo00} 
that this equation passes the Painlev\'e test, 
\cite{DK14} confirmed the existence of an appropriate gauge transformation 
of the isomonodromy system that enables to have 
the singularity reduction at the movable pole of the Laurent series.
Furthermore, they showed that the equation obtained through the singularity reduction 
is also isomonodromic with respect to the second variable of the Garnier system, 
and they derived a fourth-order ordinary differential equation that describes 
the isomonodromy deformation. 
Additionally, Dubrovin--Kapaev \cite{DK14}, as well as the earlier work of Shimomura \cite{Shimo01},
showed that the fourth-order differential equation does not have 
a Laurent series solution, but has Puiseux series solutions with a third-order branching. 
This means that the fourth-order differential equation has 
the quasi-Painlev\'e property; 
that is, every movable singular point of each solution is an algebraic branch point.
This was a crucial discovery, highlighting the potential gap between the isomonodromy 
and Painlev\'e properties. 

In this paper, in addition to the Garnier system of type $9/2$  considered by \cite{DK14}, 
we also perform a similar analysis for the fourth-order Painlev\'e equation
obtained as a restriction of the Garnier system of type $5/2+3/2$ in the sense of \cite{Kaw17}. 
These labels are related to the spectral types at irregular singular points of the associated isomonodromy systems, and the fact that they are half-integers indicates that the singularities are ramified.
Using an elementary method different from that of \cite{Shimo00}, 
we can show that the Garnier system of type $5/2+3/2$ also 
passes the Painlev\'e test in Proposition \ref{prop:Lauren-5/2}.
As our first main result, 
we prove that both of the isomonodromy systems associated with these Garnier systems 
admits the singularity reduction in Theorem \ref{thm:SR-Gar9/2} and \ref{thm:SR-Gar5/2}. 
Furthermore, we also confirm that the classical limit of the linear differential equations 
arising from the singularity reduction are generically genus 2 hyper-elliptic curves, as expected. 

Our second main result is parallel to that in \cite{DK14}. 
That is, in Theorem \ref{thm:second-IM-9/2} and \ref{thm:second-IM-5/2}, 
we show that the linear differential equations 
(\eqref{eq:SR-Gar9/2} and \eqref{eq:SR-Gar5/2})
obtained through the singularity reduction from the isomonodromy systems 
associated with the Garnier systems of type $9/2$ and $5/2+3/2$
are also isomonodromic with respect to the second variable of the Garnier systems. 
We also confirmed that there exist fourth-order ordinary differential equations 
\begin{equation} \label{eq:4th-ODE-9/2-intro}
    \hbar^2 \alpha^{(4)} + 40(\alpha')^3 \alpha''
    +36t_2 \alpha' \,\alpha'' 
    +4\alpha \, \alpha'' + 20 \left(
    \alpha' \right)^2+6t_2=0,
\end{equation}
\begin{align} 
& \hbar^2 \left( \alpha^{(4)} 
+ \frac{2 \alpha^{(3)}}{t_2} 
- \frac{4 \alpha'' \alpha^{(3)}}{\alpha'} 
-\frac{3 (\alpha'')^2}{t_2 \alpha'} 
+ \frac{3 (\alpha'')^3}{(\alpha')^2} 
\right) 
\notag \\[+.3em]
& \quad 
+ \frac{4 \alpha''}{t_2^2 \alpha'} 
+ 12 t_2 (\alpha')^3 \alpha''
+ 4 \alpha \, (\alpha')^2 \alpha'' 
+ \frac{4 \alpha\, (\alpha')^3}{t_2} 
+ 14 (\alpha')^4 
+ \frac{2}{t_2^3} 
\label{eq:4th-ODE-5/2-intro}
\end{align}
(c.f., \eqref{eq:4th-ODE-9/2} and \eqref{eq:4th-ODE-5/2}) which 
describe the isomonodromy deformation of the above linear differential equations  
obtained via the singularity reduction. 
Additionally, we find that solutions of these fourth-order equations
do not have a movable pole, but have movable branch points. 
Among these fourth-order equations, \eqref{eq:4th-ODE-9/2-intro} 
is equivalent to the example in \cite{DK14}, 
while \eqref{eq:4th-ODE-5/2-intro} provides a new example that could represent a gap between isomonodromy and the Painlev\'e property.

At first glance, our results may seem to contradict with the approaches to the Painlev\'e property via isomonodromy deformations established by 
Miwa \cite{Miwa81}, 
Malgrange \cite{Mal83a, Mal83b}, 
and the series of works by 
Inaba--Iwasaki--Saito 
\cite{IIS06a, IIS06b, IIS06c, Ina13, IS13}. 
For the case of the differential equation \eqref{eq:4th-ODE-9/2-intro} obtained from the 
Garnier system of type $9/2$, 
Dubrovin--Kapaev have briefly mentioned reasons for the exitence of such an equation  in \cite[Section 8]{DK14}, and we expect that a similar phenomenon may occur in the new equation \eqref{eq:4th-ODE-5/2-intro} as well.
On the other hand,
geometric structures of the moduli space of linear equations with ramified irregular singularities 
have also been studied in works such as \cite{BY15, Inaba22, Inaba23}. 
Interpreting our results and \cite{DK14} from geometric perspectives could lead 
to a deeper understanding of the relationship between the isomonodromy property and the Painlev\'e property.

Finally, we note that Theorems \ref{thm:SR-Gar9/2} and \ref{thm:second-IM-9/2}, 
which address the Garnier system of type 9/2, are largely covered by the results of 
Dubrovin and Kapaev \cite{DK14}. 
While their work studied the singularity reduction of the isomonodromy system 
in matrix form \eqref{eq:Lax}--\eqref{eq:Lax-tj}, 
our approach tackles the problem using the Schr\"odinger form \eqref{eq:LIV}--\eqref{eq:DIV}. 
In the work \cite{DK14}, it was necessary to apply an appropriate gauge transform 
to resolve the divergence that arises when restricting to the movable pole, 
however, one of the main assertions of our results 
(Theorem \ref{thm:SR-Gar9/2} and \ref{thm:SR-Gar5/2}) 
are that the Schr\"odinger form automatically provides such an appropriate gauge choice, 
at least when we deal with the Garnier systems of type $9/2$ and $5/2+3/2$. 
Since this insight may be valuable for analyzing the singularity reduction of 
isomonodromy systems associated with other Garnier systems, 
we have chosen to include our computational results in this paper.
In particular, \cite[Conjecture 3.1]{DK14} claims that 
``the singularity reduction exists for any isomonodromy system''. 
Our findings support this conjecture, showing that the Garnier system of type 5/2 + 3/2 
provides an example that affirms it. 
It remains an open question whether similar methods can be extended to other Garnier systems 
and the fourth-order Painlev\'e equations listed in \cite{Kaw17, KNS18}. 
Developing a general theory, in which the results presented here act as examples, 
is a promising direction for future work.

This paper is organized as follows. 
In Section \ref{sec:PIV}, as a review of the key concepts in this paper -- 
namely isomonodromy deformation, the Painlev\'e test, and singularity reduction -- 
we analyze the fourth Painlev\'e equation $P_{\rm IV}$, 
one of the second-order Painlev\'e equations, as an example. 
We will see that the biconfluent Heun equation arises as the singularity reduction 
of the isomonodromy system associated with $P_{\rm IV}$.
In Section \ref{sec:garnier}, we study the Garnier systems of type 9/2 and 5/3+3/2, 
and show that the associated isomonodromy systems 
have the singularity reduction.
Furthermore, we discuss the derivation of aforementioned
ordinary differential equations 
that describe the isomonodromic deformation of 
a linear differential equation but exhibit the quasi-Painlev\'e property.

\vspace{-.5em}
\section*{Acknowledgements}
\vspace{-.5em}

We would like to thank 
Michi-aki Inaba, 
Takuro Mochizuki, 
Akane Nakamura, 
Toshifumi Noumi, 
Yousuke Ohyama, 
Masa-Hiko Saito, 
Hidetaka Sakai,  
Yoshitsugu Takei,
and 
Kouichi Takemura 
for fruitful and insightful discussions. 
The work of K.I. is supported by JSPS KAKENHI Grand Numbers 
21H04994, 22H00094, 23K17654, 24K00525.
This work is also supported by the Research Institute for Mathematical Sciences, 
an International Joint Usage/Research Center located in Kyoto University.


\section{Singularity reduction of isomonodromy system for $\PIV$}

\label{sec:PIV}

In this section, we review the singularity reduction of the isomonodromy system, 
taking the fourth Painlev\'e equation 
\begin{equation} \label{eq:PIV}
\PIV ~:~
\hbar^2\frac{d^2q}{dt^2}
=\frac{1}{2q}
\left( \hbar\frac{dq}{dt} \right)^2+\frac{3}{2}q^3+4tq^2
+(2t^2-2\theta_\infty)q-\frac{2{\theta_0}^2}{q} 
\end{equation}
as an example. It should also be noted that a similar analysis 
was conducted in \cite{IN86, FIKN06, Mas10, DK18} for example.

Here, we deal with the Painlev\'e equation containing a formal parameter $\hbar$, 
as considered in \cite{KT98}, 
and we assume that $\hbar \ne 0$ throughout the paper.
The parameter $\hbar$ is introduced to make the classical limit of 
the associated linear differential equation, which will be introduced below, well-defined.

\subsection{Isomonodromy system for $\PIV$}

First, using the isomonodromy system 
(i.e., a pair of the linear ODE and its deformation equation) 
associated with $\PIV$, 
we briefly review what isomonodromy deformation is (see \cite{JMU81, JM81, Oka09} for general results). 
Since this paper only deals with isomonodromy deformations of second-order linear differential equations, 
we will focus on the Schr\"odinger-type equations of the form:
\begin{align}
    & \hspace{+2.em} \left( \hbar^2 \frac{\partial^2}{\partial x^2} - Q(x,t)  \right)  \psi = 0, 
    \label{eq:LIV} \\[+.5em]
    & \frac{\partial\psi}{\partial t}=
    \left( A(x,t)\frac{\partial}{\partial x} - 
    \frac{1}{2} \frac{\partial A}{\partial x}(x,t) \right) \psi, 
     \label{eq:DIV}
\end{align}
where $Q$ and $A$ are given individually for each Painlev\'e equation under consideration, 
and they are generally rational functions of $x$, 
depending analytically on other parameters within a suitable domain. 
Specifically, for the case of $\PIV$, they are given as follows  (c.f., \cite{KT98, Iwa15}):
\begin{align}
Q = Q_{\rm IV} & = \frac{\theta_0^2}{4x^2} - \frac{\theta_\infty}{4} 
    + \left( \frac{x+2t}{4} \right)^2 + \frac{H_{\rm IV}}{2x} 
    - \hbar \frac{pq}{x(x-q)} + \hbar^2 \left( - \frac{1}{4x^2} + \frac{3}{4(x-q)^2} \right), \label{eq:QIV} \\[+.5em]
    A = A_{\rm IV} & = \frac{2x}{x-q},
\end{align}
where
\begin{align}
H_{\rm IV} = 2q \left[p^2 - \frac{\hbar p}{q} 
- \left( \frac{\theta_0^2 - \hbar^2}{4q^2} - \frac{\theta_\infty}{4} 
    + \left( \frac{q+2t}{4} \right)^2 \right) \right].
\end{align}  
Through direct calculations, 
it can be confirmed that the compatibility condition of \eqref{eq:LIV}--\eqref{eq:DIV} 
can be written in terms of $Q$ and $A$ as 
\begin{equation} \label{eq:comp}
2\frac{\partial Q}{\partial t} + \hbar^2 \frac{\partial^3 A}{\partial x^3} 
- 4 Q \frac{\partial A}{\partial x} - 2 A \frac{\partial Q}{\partial x} = 0,
\end{equation}
and furthermore, this condition is found to be reduced to 
the Hamiltonian system 
\begin{equation} \label{eq:HamIV}
\hbar \frac{dq}{dt} = \frac{\partial H_{\rm IV}}{\partial p}, \quad
\hbar \frac{dp}{dt} = - \frac{\partial H_{\rm IV}}{\partial q}
\end{equation}
which is equivalent to  $\PIV$.
In this case, a nontrivial fundamental system of solutions that satisfy the system of equations 
\eqref{eq:LIV}--\eqref{eq:DIV} exists, and it follows from \eqref{eq:DIV} 
that the Stokes matrices around $x = \infty$ and the monodromy matrices around $x=0$ of \eqref{eq:LIV}
with respect to the solution do not depend on $t$. 
In this way, $P_{\rm IV}$ describes the isomonodromy deformation 
of linear differential equation \eqref{eq:LIV}. 
In other words, the monodromy data of \eqref{eq:LIV} provides conserved quantities 
of the solution to $P_{\rm IV}$. 
The connection between Painlev\'e equations and isomonodromy deformations 
highlights the integrable nature of the Painlev\'e equations.

Here, it should also be noted that the isomonodromy condition for 
a Schr\"odinger-type linear equation of the form \eqref{eq:LIV}
can always be expressed in the form of \eqref{eq:comp}.
We will use this criterion when we examine the isomonodromy condition in Section \ref{sec:garnier}.

\subsection{Laurent series solution of $P_{\rm IV}$}
\label{subsec:PIVtest}

It is well-known that, for any point $\alpha \in {\mathbb C}$, 
there exists a convergent Laurent series solution of $P_{\rm IV}$ 
at $t = \alpha$ of the form
\begin{align}
q(t)= \frac{1}{t - \alpha} \sum_{k=0}^\infty C_k(t-\alpha)^k,
\end{align}
with $C_k \in {\mathbb C}$. 
By substituting into $\PIV$, the first few coefficients are determined recursively. 
For the first term, there are two possible choices $C_0 = \pm 1$, 
but for now, let us consider the case where 
\begin{equation}
C_0 = 1.
\end{equation}
Then, the coefficients $C_1$, $C_2$ are determined as 
\begin{equation}
C_1 = - \alpha, \quad C_2 = \frac{\alpha^2 + 2 \theta_\infty - 4 \hbar}{3 \hbar}.
\end{equation}
In the determination of the next coefficients, a nontrivial phenomenon occurs. 
The linear equation that determines $C_3$ exhibits a resonance, 
making it trivially satisfied. As a result, $C_3$ can be chosen as a free parameter. 
For a later convenience, we set 
\begin{equation}
C_3 = \frac{\beta}{\hbar^2}
\end{equation}
with a free parameter $\beta$.
Then, it is easy to see that all subsequent coefficients are uniquely determined. 
Consequently, a Laurent series solution, 
which contains two free parameters $(\alpha, \beta)$, 
can be constructed as follows:
\begin{align} 
q(t) & = \frac{\hbar}{t - \alpha} - \alpha 
+ \frac{(\alpha^2 + 2 \theta_\infty - 4 \hbar)(t - \alpha) }{3 \hbar} 
+ \frac{\beta(t - \alpha)^2 }{\hbar^2}  
\notag \\[+.5em] 
& \quad + 
\frac{ (\alpha^4 + 36 \alpha \beta + 18 \theta_0^2 - 14 \theta_\infty^2 
+ 4 \alpha^2 (\theta_\infty - 5 \hbar) + 32 \theta_\infty \hbar - 26 \hbar^2)(t - \alpha)^3}{45 \hbar^3} 
\notag \\
& 
\quad + O((t-\alpha)^4).
\label{eq:LaurenPIVq}
\end{align}
We can also find that the Laurent series expansion of 
the other canonical variable $p$ in the Hamiltonian system 
\eqref{eq:HamIV} can also be obtained:  
\begin{align}
p(t) & = - \frac{\hbar}{4(t-\alpha)} - \frac{\alpha}{4} 
- \frac{(\alpha^2 - 4 \theta_\infty + 2 \hbar) (t - \alpha)}{12 \hbar} 
+ \frac{(3 \beta + 2 \alpha \theta_\infty - 2 \alpha \hbar)(t - \alpha)^2}{4 \hbar^2} 
\notag \\[+.5em] 
& \quad 
+ \frac{(\alpha^4 + 36 \alpha \beta - 72 \theta_0^2 + 16 \theta_\infty^2 
+ 4 \alpha^2 (16 \theta_\infty - 5 \hbar) - 28 \theta_\infty \hbar + 64 \hbar^2)(t_1 - \alpha)^3}{180 \hbar^3}
\notag \\
& 
\quad + O((t-\alpha)^4).
\label{eq:LaurenPIVp}
\end{align}
A proof of convergence of these series \eqref{eq:LaurenPIVq}--\eqref{eq:LaurenPIVp} 
was given in \cite[Section 4]{Oka09} for example. 
The parameters $(\alpha,\beta)$ can be regarded as free parameters
that parametrize the general solution of $\PIV$, and hencne, 
$P_{\rm IV}$ passes the Painlev\'e test. 
We note that Chiba \cite{Chi15} gave a criterion for the Painlev\'e test 
for differential equations with a weighted homogeneity 
(the case $P_{\rm IV}$ is covered in \cite{Chi16}).

\subsection{Existence of singularity reduction} 

Since $Q_{\rm IV}$, given in \eqref{eq:QIV}, explicitly depends on $q$ and $p$,  
it is nontrivial to determine whether a limit of $Q_{\rm IV}$ as $t \to \alpha$
exists when $q$ and $p$ have a pole at $t = \alpha$, as described above. 
However, through direct calculation, we can verify the following fact:

\begin{prop} \label{prop:SR-PIV}
The Schr\"odinger potential $Q_{\rm IV}$ given in \eqref{eq:QIV}, 
with $(q,p)$ substituted by the  Laurent series solution 
\eqref{eq:LaurenPIVq}--\eqref{eq:LaurenPIVp}
of the Hamiltonian system \eqref{eq:HamIV}, 
has a finite limit as $t \to \alpha$: 
\begin{equation}
\lim_{t \to \alpha} Q_{\rm IV}(x,t) = 
\left( \frac{\theta_0^2}{4x^2} - \frac{\beta + \theta_\infty \alpha}{2x} 
- \frac{\theta_\infty}{4} + \left( \frac{x + 2 \alpha}{4} \right)^2 \right)
- \hbar \frac{(x - \alpha)}{4x} - \hbar^2 \frac{1}{4x^2}.
\label{eq:lim-QIV}
\end{equation}
\end{prop}

\begin{proof}
The existence of the limit can be confirmed by a straightforward computation 
using \eqref{eq:LaurenPIVq}--\eqref{eq:LaurenPIVp}. 
\end{proof}

Let us denote by $R_{\rm IV}(x) = R_{0}(x) + \hbar R_{1}(x) + \hbar^2 R_{2}(x)$ 
the right hand side of \eqref{eq:lim-QIV}; that is, 
\begin{equation}
R_{0} = \frac{\theta_0^2}{4x^2} - \frac{2 \beta + 2 \theta_\infty \alpha}{4x} 
- \frac{\theta_\infty}{4} + \left( \frac{x + 2 \alpha}{4} \right)^2, \quad
R_{1} = - \frac{(x - \alpha)}{4x}, \quad 
R_{2}= - \frac{1}{4x^2}.
\end{equation}
Proposition \ref{prop:SR-PIV} implies that the linear equation \eqref{eq:LIV}
associated with $\PIV$ is reduced to 
\begin{equation} \label{eq:BC-Heun}
\left( \hbar^2 \frac{\partial^2}{\partial x^2} - R_{\rm IV}(x)  \right)  \psi = 0
\end{equation}
when we take the limit $t \to \alpha$. 
That is, the equation \eqref{eq:BC-Heun} is the singularity reduction 
(in the sense of \cite{DK14}) of the isomonodoromic linear ODE \eqref{eq:LIV}. 
The equation \eqref{eq:BC-Heun} is called the biconfluent Heun equation 
(c.f., \cite[Table 1]{LN21}), 
and we note that the free parameter $\beta$ is essentialy identified with 
the so-called accessory parameter.  
See also \cite{IN86, Mas10, DK18} for similar results, 
where the accessory parameters of the Heun-type equations are described 
by the free parameters that arise when passing the Painlev\'e test. 
 
As a consequence of the isomonodromy deformation, 
the monodromy data of the isomonodromic linear ODE \eqref{eq:LIV} matches 
that of the reduced linear ODE \eqref{eq:BC-Heun}. 
Therefore, analyzing the monodromy of \eqref{eq:BC-Heun} leads to a correspondence 
between the Laurent series solution \eqref{eq:LaurenPIVq} 
of the $\PIV$ and its conserved quantities. 
From the perspective of the exact WKB analysis (c.f., \cite{Vor83, KT98}), 
period integrals (so-called Voros periods) on the algebraic curve defined by
\begin{equation}
y^2 = 
R_{0} = \frac{\theta_0^2}{4x^2} - \frac{2 \beta + 2 \theta_\infty \alpha}{4x} 
- \frac{\theta_\infty}{4} + \left( \frac{x + 2 \alpha}{4} \right)^2, 
\label{eq:CL-PIV}
\end{equation}
that arises in the classical limit of equation \eqref{eq:BC-Heun},  
describes the monodromy data of \eqref{eq:BC-Heun}. 
Since the classical limit \eqref{eq:CL-PIV} is an elliptic curve, 
(with a generic choice of the parameters $\alpha, \beta, \theta_0, \theta_\infty$), 
there are two closed cycles which provides independent period integrals. 
This provides a geometric interpretation for why the dimension of 
the monodromy manifold (excluding $\theta_0$ and $\theta_\infty$ 
related to characteristic exponents) of \eqref{eq:BC-Heun} is 2, 
which aligns with the fact that $\PIV$ is a second-order differential equation. 

Since $(\alpha,\beta)$ parametrize the general solution of $\PIV$ 
through \eqref{eq:LaurenPIVq}--\eqref{eq:LaurenPIVp}, 
the explicit form of \eqref{eq:CL-PIV} can be interpreted as providing 
a correspondence between the general solution of $\PIV$ and elliptic curves.
In the following section, we will conduct a similar analysis for 
certain fourth-order Painlev\'e equations, 
where hyper-elliptic curves of genus 2 naturally emerge in place of elliptic curves.

\section{Singularity reduction of isomonodromy systems for Garnier systems}

\label{sec:garnier}

As stated in Section \ref{sec:intro}, the purpose of this paper is to investigate 
whether the singularity reduction, reviewed in the previous section, 
can be applied to the isomonodromy systems associated with fourth-order Painlev\'e equations. 
This problem has been studied by Dubrovin--Kapaev in \cite{DK14}
for the degenerate Garnier systems of type $9/2$. 
We will revisit a part of their result, and perform a similar computation 
for the degenerate Garnier systems of type $5/2+3/2$ as well (c.f., \cite{Kaw17}). 

Below, we abbreviate these Garnier systems as 
${\rm Gar_{9/2}}$ and ${\rm Gar_{5/2+3/2}}$.

\subsection{Isomonodromy systems for ${\rm Gar_{9/2}}$ and ${\rm Gar_{5/2+3/2}}$}

Both of the Garnier systems ${\rm Gar_{9/2}}$ and ${\rm Gar_{5/2+3/2}}$
describe the isomonodromy deformations of rank 2 linear system of the form
\begin{equation} \label{eq:Lax}
\hbar\frac{\partial \Phi}{\partial x} = L(x, t_1, t_2) \Phi
\end{equation}
with two independent isomonodromic times $(t_1, t_2)$, 
where $L$ is a $2\times2$ matrix whose entries are rational in $x$.
In this subsection, we adopt the matrix from \eqref{eq:Lax} 
to simplify the description of the isomonodromy system. 
However, it can be reduced to the Schr\"odinger form 
like \eqref{eq:LIV} described in the previous section 
with keeping the isomonodromic property.
This is a well-known process, but we will recall it in next subsection. 

The isomonodromy deformation of \eqref{eq:Lax} is described by 
the deformation equation of the form
\begin{equation} \label{eq:Lax-tj}
\hbar\frac{\partial \Phi}{\partial t_j} = M_j(x, t_1, t_2) \Phi
\quad (j = 1,2), 
\end{equation}
where $M_j$'s are appropriate $2\times2$ matrices with entries rational in $x$.  
The Garnier systems, describing the compatibility conditions
\begin{equation} \label{eq:comp-matrix}
\hbar \left( \frac{\partial L}{\partial t_j} - \frac{\partial M_j}{\partial x} \right) 
+ [L, M_j] = 0 
\qquad (j=1,2)
\end{equation} 
of \eqref{eq:Lax}--\eqref{eq:Lax-tj}, 
are expressed as a system of nonlinear PDEs
with respect to $(t_1, t_2)$, 
and can be formulated as a Hamiltonian system
\begin{equation} \label{eq:GarHam}
\hbar\frac{dq_i}{dt_j} =\frac{\partial H_j}{\partial p_i},\quad
\hbar\frac{dp_i}{dt_j} = -\frac{\partial H_j}{\partial q_i} \quad
(i, j = 1,2)
\end{equation}
with an appropriate Hamiltonians $H_1$ and $H_2$. 
The explicit forms of $L, M_1, M_2, H_1, H_2$ in  
\eqref{eq:Lax}--\eqref{eq:GarHam}
for our main examples are given as follows 
(c.f., \cite[Section 3.1]{Kaw17}; see also Remark \ref{rem:label-swap} below):

\begin{itemize}
\item 
For ${\rm Gar}_{9/2}$, we have
\begin{subequations}
\begin{align}
L  ~~ & = ~ L_{\rm Gar_{9/2}} ~~~\, = ~ L_{0} \, x^3 + L_{1} \, x^2 
+ L_{2} \, x + L_{3}, \label{eq:L9/2} \\[+.3em]
M_1  & = ~ M_{\rm Gar_{9/2}, 1} ~ = ~ - L_{0} \, x + M_{10}, \\[+.3em]
M_2 & = ~ M_{\rm Gar_{9/2}, 2} ~ = ~  L_{0} \, x^2 +  L_{1} \, x + M_{20}, 
\end{align}
\end{subequations}
with 
\begin{subequations}
\begin{align}
& L_0  = \begin{pmatrix} 0 & 1 \\ 0 & 0 \end{pmatrix}, \qquad
L_1 = \begin{pmatrix} 0 & p_1 \\ 1 & 0 \end{pmatrix},  \qquad
L_2 = \begin{pmatrix} q_2 & p_1^2+p_2+2t_2 \\ - p_1 & -q_2 \end{pmatrix}, \\[+.3em]
& L_3  = \begin{pmatrix} q_1 - p_1 q_2 & p_1^3 + 2p_1 p_2 - q_2^2 + t_2 p_1 - t_1 \\ 
-p_2 + t_2 & -q_1 + p_1 q_2 \end{pmatrix}, 
\\[+.3em]
& 
M_{10} = \begin{pmatrix} 0 & - 2p_1 \\ -1 & 0 \end{pmatrix},
\qquad
M_{20} = \begin{pmatrix} q_2 & p_1^2 + 2p_2 + t_2 \\ - p_1 & - q_2 \end{pmatrix}.
\end{align}
\end{subequations}
The Hamiltonians $H_1$ and $H_2$, which we describe ${\rm Gar}_{9/2}$ as \eqref{eq:GarHam}, 
are given as follows:
\begin{subequations}
\begin{align}
H_1 = H_{\mathrm{Gar_{9/2}},1}
& ={p_1}^4+3{p_1}^2p_2+p_1{q_2}^2-2q_1q_2+{p_2}^2-t_1p_1+t_2p_2, \\[+.3em]
H_2 = H_{\mathrm{Gar_{9/2}},2}
& = -{p_1}^3p_2+{p_1}^2{q_2}^2 +t_2{p_1}^3-2p_1q_1q_2-2p_1{p_2}^2+p_2{q_2}^2\notag \\ 
& \quad +t_2p_1p_2+{q_1}^2-t_2{q_2}^2+{t_1}^2p_1+t_1p_2.
\end{align}  
\end{subequations}

\newpage
\item 
For ${\rm Gar}_{5/2+3/2}$, we have
\begin{subequations}
\begin{align}
L  ~~ & = ~ L_{\rm Gar_{5/2+3/2}} ~~~\, = ~ L_{0} \, x + L_{1} + \frac{L_{2}}{x} + \frac{L_{3}}{x^2}, 
\label{eq:L5/2+3/2} \\[+.3em]
M_1  & = ~ M_{\rm Gar_{5/2+3/2}, 1} ~ = ~ M_{10} \, x + M_{11}, \\[+.3em]
M_2 & = ~ M_{\rm Gar_{5/2+3/2}, 2} ~ = ~ M_{20} - \frac{L_3}{t_2 x}, 
\end{align}
\end{subequations}
with 
\begin{subequations}
\begin{align}
& L_0  = \begin{pmatrix} 0 & 1 \\ 0 & 0 \end{pmatrix}, \quad
L_1 = \begin{pmatrix} q_1 & p_1 - q_1^2 - t_1 \\ 1 & - q_1 \end{pmatrix},  
\\[+.3em] & 
L_2  = \begin{pmatrix} p_2 q_2 & q_2 \\ -p_1 & -p_2 q_2  \end{pmatrix}, \quad
L_3 = \begin{pmatrix} 0 & 0 \\ {t_2}/{q_2} & 0 \end{pmatrix}, 
\\[+.3em]
& 
M_{10} = \begin{pmatrix} 0 & -1 \\ 0 & 0 \end{pmatrix},
\quad
M_{11} = \begin{pmatrix} - q_1 & 0 \\ -1 & q_1 \end{pmatrix},
\quad
M_{20} = \begin{pmatrix} 0 &-q_2/t_2 \\ 0 & 0 \end{pmatrix}.
\end{align}
\end{subequations}
The Hamiltonians $H_1$ and $H_2$, which we describe ${\rm Gar}_{5/2+3/2}$ as \eqref{eq:GarHam},  
are given as follows:
\begin{subequations}
\begin{align}
H_1 = H_{\mathrm{Gar_{5/2+3/2}},1}
& = p_1^2 - (q_1^2 + t_1) p_1 - 2 p_2 q_1 q_2 - q_2 - \frac{t_2}{q_2}, 
\label{eq:Ham1-Gar-5/2}\\[+.3em]
H_2 = H_{\mathrm{Gar_{5/2+3/2}},2}
& = 
\frac{p_2^2 q_2^2 - p_1 q_2}{t_2} + \frac{p_1 - q_1^2 - t_1}{q_2}. 
\label{eq:Ham2-Gar-5/2}
\end{align}  
\end{subequations}

\end{itemize}

The classification of fourth-order Painlev\'e equations
has been relatively recently developed in works such as \cite{Kaw17, KNS18}, 
but the some of equations had already been introduced earlier 
from the perspective of generalizing the Painlev\'e equations. 
For example, \cite{Kim89} has already introduced the
Garnier system of type $9/2$ which we will analyze in Section \ref{subsec:Lax-Gar}.
A priori, it is not guaranteed that the Hamiltonian system \eqref{eq:GarHam} 
can be reduced to a fourth-order differential equation for an unknown function, 
so, it might be more appropriate to refer 
to this system as a four-dimensional Painlev\'e equation following \cite{KNS18}.

\begin{rem} \label{rem:label-swap}
We follow the approach of \cite{Nak17} and swap the labels of 
the isomonodromic times of the Garnier systems of type $9/2$ in \cite{Kaw17}. 
Specifically, the variable $t_1$ (resp., $t_2$) in \cite{Kaw17} 
is identical to $t_2$ (resp., $t_1$) in this paper. 
The reason for swapping the labels in this way is to facilitate 
a parallel discussion for the two examples in the following sections. 
This choice is also related to the degrees of Hamiltonians 
with respect to the weighted homogeneity that the equations possess
(c.f., \cite{Nak20, Chi24}). 
\end{rem}

\subsection{From rank 2 system to Schr\"odinger form}
\label{subsec:Lax-vs-Sch}

Here we will briefly review how to obtain the isomonodromy system 
in the Schr\"odinger form \eqref{eq:LIV}-\eqref{eq:DIV}
from the rank 2 system \eqref{eq:Lax}-\eqref{eq:Lax-tj}. 

Suppose we have given a rank 2 system of the form \eqref{eq:Lax} 
with a $2\times2$ matrix $L = (L_{ij})_{i,j =1,2}$.
Then, we can verify that the first entry $\phi_1$ of 
$\Phi = {}^{t}(\phi_1, \phi_2)$
satisfies the following scaler ODE: 
\begin{equation} \label{eq:scalar-ODE-gl2}
\left(\hbar^2 \frac{\partial^2}{\partial x^2} 
+ P_1(x,t) \hbar \frac{\partial}{\partial x}
+ P_2(x,t) \right) \phi_1 = 0,
\end{equation}
where 
\begin{subequations}
\begin{align}
P_1  & = - L_{11} - L_{12} 
- \frac{\hbar}{L_{12}} \frac{\partial L_{12}}{\partial x}, \\[+.3em]
P_2 & = L_{11} L_{22} - L_{12} L_{21} 
+ \hbar \left(- \frac{\partial L_{11}}{\partial x} + 
\frac{L_{11}}{L_{12}}\,\frac{\partial L_{12}}{\partial x} \right).
\end{align} 
\end{subequations}
As is well-known, the zeros of $L_{12}$ gives the so-called apparent singular points
of \eqref{eq:scalar-ODE-gl2}. 
Further applying the gauge transform 
\begin{equation} \label{eq:gauge}
\phi_1 = G \psi, \quad 
G = \exp\left( - \frac{1}{2\hbar} \int P_1(x) \, dx \right), 
\end{equation}
the equation \eqref{eq:scalar-ODE-gl2} is converted into 
the Schr\"odinger form \eqref{eq:LIV} 
with the potential function
\begin{equation} \label{eq:Sch-pot}
Q = - P_2 + \frac{P_1^2}{4} + \frac{\hbar}{2} \frac{\partial P_1}{\partial x}.
\end{equation}

On the other hand, if $\Phi$ also satisfies the deformation equation \eqref{eq:Lax-tj}, 
then $\phi_1$ satisfies the PDE
\begin{equation} \label{eq:Lax-tj-gl2}
\hbar \frac{\partial \phi_1}{\partial t_j} 
= M_{j,11} \phi_1 + M_{j,12} \phi_2 
= \frac{\hbar M_{j,12}}{L_{12}} \frac{\partial \phi_1}{\partial x}
+ \left( M_{j,11} - \frac{M_{j,12} L_{11}}{L_{12}} \right) \phi_1, 
\end{equation} 
where $M_{j, kl}$ is the $(k,l)$-entry of the matrix $M_j$. 
Here we have also used \eqref{eq:Lax}
to express $\phi_2$ in terms of $\partial \phi_1/\partial x$. 
Then, the gauge transform \eqref{eq:gauge} converts \eqref{eq:Lax-tj-gl2} 
to 
\begin{equation} \label{eq:Lax-tj-sl2-pre}
\hbar \frac{\partial \psi}{\partial t_j} 
= \frac{\hbar M_{j,12}}{L_{12}} \frac{\partial \psi}{\partial x}
+ \left( M_{j,11} - \frac{M_{j,12} L_{11}}{L_{12}}
- \frac{M_{j,12} P_1}{2L_{12}} 
- \frac{\hbar}{G} \frac{\partial G}{\partial t_j}  \right) \psi.
\end{equation} 
Then, if we take an appropriate normalization of $G$ given in \eqref{eq:gauge} 
by multiplying it with a function of $t_j$
which does not depend on $x$, then the 
equation \eqref{eq:Lax-tj-sl2-pre} can be reduced to
\begin{equation} \label{eq:Lax-tj-sl2}
\frac{\partial\psi}{\partial t_j}=
    \left( A_j\frac{\partial}{\partial x} - 
    \frac{1}{2} \frac{\partial A_j}{\partial x} \right) \psi
    \end{equation} 
with 
\begin{equation} \label{eq:GarAj}
A_j = \frac{M_{j,12}}{L_{12}}. 
\end{equation}
Thus we have obtained the Schr\"odinger form \eqref{eq:LIV}--\eqref{eq:DIV}.


We denote by 
$Q_{\rm Gar_{9/2}}$ and $Q_{\rm Gar_{5/2+3/2}}$ 
the Schr\"odinger potential \eqref{eq:Sch-pot}
obtained from the isomonodromy linear system \eqref{eq:Lax}--\eqref{eq:Lax-tj}
associated with Garnier systems ${\rm Gar}_{9/2}$ and ${\rm Gar}_{5/2+3/2}$
by the above procedure, respectively. 
We also denote by $A_{{\rm Gar_{9/2}}, j}$ and $A_{{\rm Gar_{5/2+3/2}}, j}$ 
the function \eqref{eq:GarAj} defined in a similar manner. 
It is possible to explicitly write down the expressions of these functions, 
but due to their complexity, we will omit them here.
In the next section, we will investigate the behavior 
of these functions under singularity reduction.

\subsection{Analysis of ${\rm Gar_{9/2}}$}
\label{subsec:Lax-Gar}

As mentioned above, the singularity reduction at the movable poles 
of the Garnier system of type $9/2$ has been discussed in the previous work \cite{DK14} 
of Dubrovin--Kapaev. 
Here, we revisit their result through the analysis in the Schr\"odinger form, 
and recall their important observation, 
which highlights the existence of a differential equation that does not exhibit 
the Painlev\'e property, despite describing the isomonodromy deformation.


\subsubsection{Laurent series solution of ${\rm Gar}_{9/2}$}
\label{subsec:Laurent-9/2}

We begin by deriving the Laurent series solution for the Garnier system of type $9/2$.
Our example involves two isomonodromic time variables $(t_1, t_2)$, 
but for now, let's treat the equation as a system of ODEs 
with respect to $t_1$, and proceed with the analysis in a manner similar to Section \ref{subsec:PIVtest}.

First let us recall the following result of Shimomura.

\begin{prop}[{\cite[Theorem C]{Shimo00}}] \label{prop:Lauren-9/2}
When considering the Garnier system of type $9/2$ with $t_2$ restricted to a constant 
and treating it as a system of ODEs 
with respect to $t_1$, 
there exists a convergent Laurent series solution of the following form:
\begin{subequations}
\begin{align}
q_1(t_1) & 
= -\frac{\hbar^5}{(t_1 - \alpha)^5} 
+ \frac{\beta \hbar^3}{(t_1 - \alpha)^3} + \gamma 
+ \frac{(10 \alpha - 18 t_2 \beta - 35 \beta^3)(t_1 - \alpha)}{70 \hbar} 
\notag \\[+.5em]
& \quad 
- \frac{3 (50 \beta \gamma - 3 \hbar) (t_1 - \alpha)^2}{20 \hbar^2} 
+ \frac{\delta(t_1 - \alpha)^3 }{\hbar^3} 
- \frac{(-36 t_2 \gamma + 105 \beta^2 \gamma - 5 \beta \hbar)(t_1 - \alpha)^4}{14 \hbar^4} 
\notag \\[+.5em]
& \quad + O((t_1-\alpha)^5),
\label{eq:LaurenGar9/2q1}
\\[+.5em] 
q_2(t_1) & = 
-\frac{\hbar^3}{(t_2 - \alpha)^3} 
-\frac{3(4 t_2 + 5 \beta^2)(t_1 - \alpha)}{20 \hbar} 
- \frac{6 \gamma (t_2 - \alpha)^2}{\hbar^2} 
\notag \\[+.5em]
& \quad
+ \frac{(4 \alpha - 24 t_2 \beta - 35 \beta^3)(t_1 - \alpha)^3}{14 \hbar^3} 
+ \frac{(-30 \beta \gamma + \hbar)(t_1 - \alpha)^4}{4 \hbar^4} 
\notag \\[+.5em] 
& \quad+ \frac{3 (1008 t_2^2 + 400 \alpha \beta + 120 t_2 \beta^2 - 1925 \beta^4 + 1400 \delta)(t_1 - \alpha)^5 }
{15400 \hbar^5}
+ O((t_1-\alpha)^6),
\label{eq:LaurenGar9/2q2}
\\[+.5em] 
p_1(t_1) & = \frac{\hbar^2}{(t_1 - \alpha)^2} 
+ \frac{\beta}{2} 
-\frac{3 (4 t_2 + 5 \beta^2)(t_1 - \alpha)^2}{20 \hbar^2} 
- \frac{4 \gamma (t_1 - \alpha)^3}{\hbar^3} 
\notag \\[+.5em]
& \quad
+ \frac{(4 \alpha - 24 t_2 \beta - 35 \beta^3)(t_1 - \alpha)^4 }{28 \hbar^4} 
+ \frac{(-30 \beta \gamma + \hbar)(t_1 - \alpha)^5}{10 \hbar^5}
\notag \\[+.5em]
& \quad
+ \frac{(1008 t_2^2 + 400 \alpha \beta + 120 t_2 \beta^2 - 1925 \beta^4 + 1400 \delta) (t_1 - \alpha)^6}
{15400 \hbar^6}
+ O((t_1-\alpha)^7),
\label{eq:LaurenGar9/2p1}
\\[+.5em] 
p_2(t_1) & = -\frac{3 \beta \hbar^2}{2 (t_1 - \alpha)^2} 
+ \left( t_2 + \frac{3 \beta^2}{2} \right) 
+ \frac{6 \gamma (t_1 - \alpha)}{\hbar} 
+ \frac{9(4 t_2 \beta + 5 \beta^3)(t_1 - \alpha)^2 }{40 \hbar^2} 
\notag \\[+.5em]
& \quad
+ \frac{(t_1 - \alpha)^3}{5 \hbar^2} 
- \frac{3(1008 t_2^2 + 400 \alpha \beta + 120 t_2 \beta^2 - 1925 \beta^4 - 1680 \delta)(t_1 - \alpha)^4}
{12320 \hbar^4} 
\notag \\[+.5em]
& \quad
- \frac{9 (32 t_2 \gamma + 70 \beta^2 \gamma - \beta \hbar)(t_1 - \alpha)^5}{140 \hbar^5} 
+ O((t_1-\alpha)^6),
\label{eq:LaurenGar9/2p2}
\end{align}
\end{subequations}
with free parameters $(\alpha, \beta, \gamma, \delta)$ independent of $t_1$. 
Namely, the restriction of the Garnier system of type $9/2$ 
passes the Painlev\'e test. 
\end{prop}
\begin{proof}
See \cite[Theorem C]{Shimo00}.
\end{proof}

The first few coefficients of this Laurent series solution agrees with the results obtained 
by \cite[Section 3.1]{Nak17}, where the autonomous-version of the Garnier equation is considered
(see also \cite{Nak20}).

\subsubsection{Existence of singularity reduction $SR_{\rm Gar_{9/2}}$}

Recall that $Q_{\rm Gar_{9/2}}$ and $A_{{\rm Gar_{9/2}},j}$
are defined in \eqref{eq:Sch-pot} and \eqref{eq:GarAj}
obtained from the ismonodormy system associated with ${\rm Gar}_{9/2}$. 
As in the case of second-order Painlev\'e equations $\PIV$, 
it can be confirmed that the divergence of the Laurent series solution at the movable pole is offset when considering the associated isomonodromy system, 
leading to a finite limit.

\begin{thm}[{c.f., \cite[Theorem 3.1]{DK14}}] 
\label{thm:SR-Gar9/2} 
The functions $Q_{\rm Gar_{9/2}}$ and $A_{{\rm Gar_{9/2}},2}$, 
with $(q_1,q_2,p_1,p_2)$ substituted by the Laurent series solution 
\eqref{eq:LaurenGar9/2q1}--\eqref{eq:LaurenGar9/2p2}
of the Hamiltonian system \eqref{eq:GarHam}, 
have finite limits as $t_1 \to \alpha$: 
\begin{align}
\lim_{t_1 \to \alpha} Q_{\rm Gar_{9/2}}(x,t_1, t_2)  
& = R_{\rm Gar_{9/2}}(x,t_2) \notag \\
& = R_{\rm Gar_{9/2},0}(x,t_2) + \hbar R_{{\rm Gar_{9/2}},1}(x,t_2) 
+ \hbar^2 R_{{\rm Gar_{9/2}},2}(x,t_2),
\label{eq:lim-QGar9/2}
\\
\lim_{t_1 \to \alpha} A_{{\rm Gar_{9/2}},2}(x,t_1, t_2) 
&  = B_{\rm Gar_{9/2}}(x,t_2) = \frac{2}{2x-3\beta},
\label{eq:BGar9/2}
\end{align}
where 
\begin{subequations}
\begin{align}
R_{\rm Gar_{9/2},0} & = 
x^5 + 3 t_2 x^3 - \alpha x^2 
+ \frac{3 (9072 t_2^2 + 8000 \alpha \beta - 34560 t_2 \beta^2 - 51975 \beta^4 - 15120 \delta)}{12320}\,x
\notag \\[+.5em]
& \quad 
-\frac{9 (9072 t_2^2 \beta + 1840 \alpha \beta^2 - 6840 t_2 \beta^3 
- 31185 \beta^5 - 221760 \gamma^2 - 15120 \beta \delta)}{24640}, 
\label{eq:RG9/2-0}\\[+.5em]
R_{\rm Gar_{9/2},1} & = \frac{18 \gamma}{2x-3\beta}, 
\label{eq:RG9/2-1} \\[+.5em]
R_{\rm Gar_{9/2},2} & = \frac{3}{(2x-3\beta)^2}. 
\label{eq:RG9/2-2}
\end{align}
\end{subequations}
\end{thm}
\begin{proof}
The existence of the limit can be confirmed by a straightforward computation 
using \eqref{eq:LaurenGar9/2q1}--\eqref{eq:LaurenGar9/2p2}. 
\end{proof}

\begin{rem}
Theorem \ref{thm:SR-Gar9/2} is essentially shown in \cite[Theorem 3.1]{DK14}. 
While \cite{DK14} describes the isomonodromy system in matrix form, 
we present it as the Schr\"odinger form. 
In matrix form, it is necessary to find an appropriate gauge transformation 
for the singularity reduction to converge (c.f., \cite[eq.\,(3.1)]{DK14}), 
but at least in this example (as well as in $P_{\rm IV}$ and ${\rm Gar}_{5/2+3/2}$), 
the Schr\"odinger form automatically provides such a gauge choice.
\end{rem}

Theorem \ref{thm:SR-Gar9/2} implies that, 
when taking the limit $t_1 \to \alpha$, 
the singularity reduction
\begin{equation} \label{eq:SR-Gar9/2}
SR_{\rm Gar_{9/2}}~:~ 
\left( \hbar^2 \frac{\partial^2}{\partial x^2} - R_{\rm Gar_{9/2}}(x, t_2)  \right)  \psi = 0
\end{equation}
of the isomonodromy system associated with the Garnier system 
of type $9/2$ exists.
Moreover, for generic choice of the free parameters $(\alpha,\beta,\gamma,\delta)$,
the classical limit 
\begin{equation}
y^2 = R_{\rm Gar_{9/2},0}(x, t_2)
\label{eq:CL-Gar9/2}
\end{equation}
of $SR_{\rm Gar_{9/2}}$ defines a family of hyper-elliptic curves of genus 2 parametrized by $t_2$. 


\subsubsection{Isomonodromy property of $SR_{\rm Gar_{9/2}}$ and quasi-Painlev\'e property}

It follows from \eqref{eq:RG9/2-0}--\eqref{eq:RG9/2-2} that 
the the regular singular point $x = 3\beta/2$ of
the Schr\"odinger-type equation $SR_{\rm Gar_{9/2}}$  
is an apparent singular point. 
Furthermore, it can be shown that by appropriately choosing 
the aforementioned parameters $(\alpha,\beta,\gamma,\delta)$ 
as approproate functions of $t_2$, 
the equation $SR_{\rm Gar_{9/2}}$  
defines an isomonodromy family. 
The precise statement is given as follows.

\begin{thm}[{c.f., \cite[Theorem 3.2]{DK14}}]
\label{thm:second-IM-9/2}
If the parameters $(\alpha,\beta,\gamma,\delta)$ are functions of $t_2$ satisfying 
the system of ODEs
\begin{subequations}
\begin{align}
\hbar \frac{d\alpha}{dt_2} & = - \frac{3\beta\hbar}{2}, \label{eq:SP-9/2-1}
\\[+.5em]
\hbar \frac{d\beta}{dt_2} & = - {12 \gamma},
\\[+.5em]
\hbar\frac{d\gamma}{dt_2} & = -\frac{9 (336 t_2^2 - 160 \alpha \beta 
+ 1800 t_2 \beta^2 + 1925 \beta^4 - 560 \delta)}{12320},
\\[+.5em]
\hbar\frac{d\delta}{dt_2} & = 
\frac{103680 t_2 \beta \gamma -12000 \alpha \gamma 
+  311850 \beta^3 \gamma + 728 t_2 \hbar - 4665 \beta^2 \hbar}{1890}, 
\label{eq:SP-9/2-4}
\end{align}
\end{subequations}
then $R = R_{\rm Gar_{9/2}}$ satisfies 
\begin{equation} \label{eq:comp-SR-9/2}
2\frac{\partial R}{\partial t_2} + \hbar^2 \frac{\partial^3 B}{\partial x^3} 
- 4 R \frac{\partial B}{\partial x} - 2 B \frac{\partial R}{\partial x} = 0,
\end{equation}
with the choice $B = B_{\rm Gar_{9/2}}$ given in \eqref{eq:BGar9/2}. 
Therefore, $SR_{\rm Gar_{9/2}}$  
with $(\alpha, \beta, \gamma, \delta)$ 
satisfying \eqref{eq:SP-9/2-1}--\eqref{eq:SP-9/2-4}  
is compatible with the PDE 
\begin{equation} \label{eq:Lax-SR-9/2}
\frac{\partial\psi}{\partial t_2}=
    \left( B_{\rm Gar_{9/2}} \frac{\partial}{\partial x} - 
    \frac{1}{2} \frac{\partial B_{\rm Gar_{9/2}}}{\partial x} \right) \psi,
\end{equation}
and thus $SR_{\rm Gar_{9/2}}$ forms an isomonodromic family of linear ODEs parametrized by $t_2$.
\end{thm}

Since the isomonodromy condition \eqref{eq:comp-SR-9/2} 
can be verified through a straightforward calculation, we omit the proof.

\begin{rem} \label{rem:finding-secondary-isomonodromy}
The system \eqref{eq:SP-9/2-1}--\eqref{eq:SP-9/2-4} 
of ODEs can be obtained as follows: 
First, we require that the Laurent series solution 
\eqref{eq:LaurenGar9/2q1}--\eqref{eq:LaurenGar9/2p2} 
constructed in Proposition \ref{prop:Lauren-9/2} 
simultaneously satisfies a Hamiltonian system with $t_2$ as the independent variable
(i.e., the equation \eqref{eq:GarHam} for $j = 2$). 
In Proposition \ref{prop:Lauren-9/2}, it was sufficient for 
$(\alpha,\beta,\gamma,\delta)$
to be constants independent of $t_1$, but by requiring them to depend on $t_2$, 
a differential equation for the coefficients of the Laurent series 
\eqref{eq:LaurenGar9/2q1}--\eqref{eq:LaurenGar9/2p2} 
with respect to $t_2$ is derived. 
The above \eqref{eq:SP-9/2-1}--\eqref{eq:SP-9/2-4} was discovered in this manner.
\end{rem}

The system \eqref{eq:SP-9/2-1}--\eqref{eq:SP-9/2-4} 
is reduced to the following fourth-order single ODE for $\alpha$:
\begin{equation} \label{eq:4th-ODE-9/2}
    \hbar^2 \alpha^{(4)} + 40(\alpha')^3 \alpha''
    +36t_2 \alpha' \,\alpha'' 
    +4\alpha \, \alpha'' + 20 \left(
    \alpha' \right)^2+6t_2=0,
\end{equation}
where $'$ means the derivative with respect to $t_2$.
We can verify that, through a certain rescaling of variables, 
the equation \eqref{eq:4th-ODE-9/2} is equivalent to the equation 
\cite[eq.\,(1.6)]{Shimo00} and \cite[eq.\,${\rm P}_{\rm I}^{(2,1)}$]{DK14}.
In these works, it was shown that the equation has the quasi-Painlev\'e property;
that is, every movable singular point of each solution is an algebraic branch point.
In fact, when $b$ is an arbitrary point, \eqref{eq:4th-ODE-9/2} has a solution 
that is expressed as a Puiseux series of the following form:
\begin{align}
\alpha(t_2) 
& = 
c_1 - 3^{1/3} \hbar^{2/3} (t_2 - b)^{1/3} 
+ \frac{9 \cdot 3^{2/3} b (t_2 - b)^{5/3}}{35 \hbar^{2/3}} 
- \frac{3^{1/3} c_1 (t_2 - b)^{7/3}}{7 \hbar^{4/3}} 
\notag \\[+.3em]
& \quad 
+ \frac{9 \cdot 3^{2/3} (t_2 - b)^{8/3}}{20 \hbar^{2/3}} 
+ \frac{c_2 (t_2 - b)^3}{\hbar^2} 
+ \frac{c_3 (t_2 - b)^{11/3}}{3^{1/3} \hbar^{8/3}} 
- \frac{729 b (t_2 - b)^4}{980 \hbar^2} 
\notag \\[+.3em]
& \quad 
- \frac{3^{1/3} (729 b^3 + 1372 c_1^2 + 23814 b c_2) (t_2 - b)^{13/3}}{22295 \hbar^{10/3}}
+ O((t_2-b)^{{16}/{3}}), 
\label{eq:PSer-9/2}
\end{align}
where $(b, c_1, c_2, c_3)$ can be taken as free parameters.
Since the number of free parameters in the Puiseux series 
is equal to the order of ODEs, we can say that 
\eqref{eq:4th-ODE-9/2} passes ``quasi-Painlev\'e test''.

According to Theorem \ref{thm:second-IM-9/2}, equation \eqref{eq:4th-ODE-9/2} 
(or the system \eqref{eq:SP-9/2-1}--\eqref{eq:SP-9/2-4}) 
describes the isomonodromy deformation of $SR_{\rm Gar_{9/2}}$, 
but as we have seen here, 
it does not possess the Painlev\'e property in usual sense. 
This phenomenon was first observed in \cite{DK14}. 

Furthermore, $R_{\rm Gar_{9/2}}$ with $\alpha$ 
substituted by the Puiseux series \eqref{eq:PSer-9/2}
also has a finite limit as $t_2 \to b$:
\begin{equation}
\lim_{t_2 \to b} R_{\rm Gar_{9/2}}(x,t_2) = 
x^5 + 3 b x^3 - c_1 x^2 + \frac{8748 b^2 + 14553 c_2}{3969} x - \frac{4347 b c_1 - 7007 c_3}{3969}.
\end{equation}
This means that $SR_{\rm Gar_{9/2}}$ also admits 
the second singularity reduction as $t_2 \to b$, 
and the classical limit of the reduced Schr\"odinger-type ODE 
defines a hyper-elliptic curve of genus 2 if we choose generic $(b, c_1, c_2, c_3)$.

\begin{rem} \label{rem:another-Laurent-9/2}
The details are omitted here, but through a case-by-case analysis, 
we find that there are several possible pole orders and coefficients of initial term 
in the Laurent series solutions of $\rm Gar_{9/2}$. 
The above Laurent series solution \eqref{eq:LaurenGar9/2q1}--\eqref{eq:LaurenGar9/2p2} 
is one of the choice.  
There is another Laurent series solution of ${\rm Gar}_{9/2}$ of the form 
\begin{subequations}
\begin{align}
q_1(t_1) & = 
\frac{9 \hbar^5}{(t_1 - \alpha)^5} 
+ \frac{\alpha (t_1 - \alpha)}{21 \hbar} 
+ \frac{(t_1 - \alpha)^2}{20 \hbar} 
+ \frac{\beta (t_1 - \alpha)^3}{\hbar^3} 
+ \frac{\gamma (t_1 - \alpha)^5}{\hbar^5}
+ O((t_1-\alpha)^6),
\label{eq:LaurenGar9/2q1alt}\\[+.3em]
q_2(t_1) & = 
\frac{3 \hbar^2}{(t_1 - \alpha)^2} 
- \frac{3 t_2 (t_1 - \alpha)^2}{35 \hbar^2} 
- \frac{\alpha (t_1 - \alpha)^4}{63 \hbar^4} 
- \frac{(t_1 - \alpha)^5}{30 \hbar^4} 
+ O((t_1-\alpha)^6),
\\[+.3em]
p_1(t_1) & = 
\frac{3 \hbar^2}{(t_1 - \alpha)^2} 
- \frac{3 t_2 (t_1 - \alpha)^2}{35 \hbar^2} 
- \frac{\alpha (t_1 - \alpha)^4}{63 \hbar^4} 
- \frac{(t_1 - \alpha)^5}{30 \hbar^4} 
+ O((t_1-\alpha)^6), 
\\[+.3em]
p_2(t_1) & = 
-\frac{9 \hbar^4}{(t_1 - \alpha)^4} 
+ \frac{8 t_2}{35} 
+ \frac{2 \alpha (t_1 - \alpha)^2}{21 \hbar^2} 
+ \frac{2 (t_1 - \alpha)^3}{15 \hbar^2} 
+ O((t_1-\alpha)^4),
\label{eq:LaurenGar9/2p2alt}
\end{align}
\end{subequations}
which contains three free parameters $(\alpha, \beta, \gamma)$.
See \cite[Section 3.1]{Nak17} for a related computation in autonomous case. 
We can verify that the above Laurent series solution also  provides 
the singularity reduction of the isomonodromy system for ${\rm Gar_{9/2}}$:
\begin{equation}
\lim_{t_1 \to \alpha} Q_{\rm Gar_{9/2}}
= x^5 + 3 t_2 x^3 - \alpha x^2 
+ \frac{9(153 t_2^2 + 2695 \beta) x}{490} 
- \frac{3(134 \alpha t_2 - 5005 \gamma)}{280}.
\end{equation}
However, in this case, a result similar to Theorem \ref{thm:second-IM-9/2} 
cannot be obtained using the same method. 
Even if the above Laurent series \eqref{eq:LaurenGar9/2q1alt}--\eqref{eq:LaurenGar9/2p2alt}
are substituted into the Hamiltonian system with 
$t_2$ as the time variable, inconsistencies arise in the differential relations 
for the parameters $(\alpha, \beta, \gamma)$, and a system of equations like 
\eqref{eq:SP-9/2-1}--\eqref{eq:SP-9/2-4} cannot be obtained. 
It may be a meaningful problem to consider which Laurent series solution 
could yield results similar to Theorem \ref{thm:second-IM-9/2}. 
\end{rem}

\subsection{Analysis of ${\rm Gar_{5/2 + 3/2}}$}

In this section, we will show that, by applying the method described 
in the previous section to ${\rm Gar_{5/2+3/2}}$, 
results parallel to those of \cite{DK14} can be obtained.
Unlike the case of $\rm Gar_{9/2}$, the Hamiltonians  \eqref{eq:Ham1-Gar-5/2}--\eqref{eq:Ham2-Gar-5/2} for ${\rm Gar_{5/2+3/2}}$ are rational functions. Therefore, in this section, we assume that $t_2$ and $q_2$ take nonzero values.

\subsubsection{Laurent series solution of ${\rm Gar}_{5/2 + 3/2}$}
We begin by deriving the Laurent series solution for the Garnier system of type $5/2+3/2$, 
similarly to Section \ref{subsec:Laurent-9/2}. 

\begin{prop} \label{prop:Lauren-5/2}
When considering the Garnier system of type $5/2+3/2$ with $t_2$ restricted to a constant 
and treating it as as a system of ODEs 
with respect to $t_1$, 
there exists a convergent Laurent series solution of the following form:
\begin{subequations}
\begin{align}
q_1(t_1) & = 
\frac{\hbar}{t_1 - \alpha} - \frac{(\alpha - 2\beta)(t_1 - \alpha)}{3\hbar} 
+ \frac{\gamma (t_1 - \alpha)^2}{\hbar^2} 
+ \frac{\delta (t_1 - \alpha)^3}{\hbar^3} 
\notag \\[+.3em]
& \quad 
- \frac{(3\alpha \gamma - 10\beta \gamma + \alpha \hbar - 2\beta \hbar)(t_1 - \alpha)^4}{9\hbar^4}
+ O((t_1-\alpha)^5),
\label{eq:LaurenGar5/2q1}
\\[+.5em] 
q_2(t_1) & = \frac{\beta \hbar^2}{(t_1 - \alpha)^2} + \frac{1}{3}(\alpha \beta - 2 \beta^2) 
- \frac{2 \beta \gamma (t_1 - \alpha)}{3 \hbar} 
+ \frac{(\alpha^2 \beta - 4 \alpha \beta^2 + 4 \beta^3 - 9 \beta \delta)(t_1 - \alpha)^2}{18 \hbar^2}
\notag \\[+.3em]
& \quad 
+ \frac{2 (-2 \alpha \beta \gamma + \alpha \beta \hbar - 2 \beta^2 \hbar)(t_1 - \alpha)^3}{45 \hbar^3}
+ O((t_1-\alpha)^4),
\label{eq:LaurenGar5/2q2}
\\[+.5em] 
p_1(t_1) & = \beta + \left(\frac{1}{2} + \frac{2 \gamma}{\hbar} \right)(t_1 - \alpha) 
+ \frac{(\alpha^2 - 4 \alpha \beta + 4 \beta^2 + 45 \delta)(t_1 - \alpha)^2}{18 \hbar^2} 
\notag \\[+.3em]
& \quad
- \frac{(4 \alpha \gamma - 12 \beta \gamma + \alpha \hbar - 2 \beta \hbar)(t_1 - \alpha)^3}{3 \hbar^3}
+ O((t_1-\alpha)^4),
\label{eq:LaurenGar5/2p1}
\\[+.5em] 
p_2(t_1) & = -\frac{(t_1 - \alpha)}{\hbar} 
- \frac{(4 \gamma + \hbar)(t_1 - \alpha)^2}{4 \beta \hbar^2} 
+ \frac{2(\alpha - 2 \beta)(t_1 - \alpha)^3}{3 \hbar^3}
\notag \\[+.3em]
& \quad
+ \frac{(4 \alpha \gamma - 20 \beta \gamma + \alpha \hbar - 2 \beta \hbar)(t_1 - \alpha)^4}{12 \beta \hbar^4}
+ O((t_1-\alpha)^5),
\label{eq:LaurenGar5/2p2}
\end{align}
\end{subequations}
with free parameters $(\alpha, \beta, \gamma, \delta)$ independent of $t_1$, 
but we assume $\beta \ne 0$. 
Namely, the restriction of the Garnier system of type $5/2+3/2$ 
passes the Painlev\'e test. 
\end{prop}

\begin{proof}
We can verify that there exists Laurent series of the form 
\eqref{eq:LaurenGar5/2q1}--\eqref{eq:LaurenGar5/2p2}
including the free parameters and 
satisfing ${\rm Gar}_{5/2+3/2}$ by direct computation. 
To prove the convergence of the Laurent series 
obtained from the Painlev\'e test, 
we follow a method of \cite{HY00}. 
If we take new unknown functions as 
\begin{equation} \label{eq:mirror-transform}
(\xi_1, \xi_2, \xi_3, \xi_4) = 
\left( \frac{1}{q_1}, 
\frac{q_2}{q_1^2}, 
4 q_2 p_2 + \frac{4q_2}{q_1}, 
p_1q_1^2+q_2+2p_2 q_1 q_2
\right),
\end{equation}
then the restriction of ${\rm Gar_{5/2+3/2}}$ is transformed to 
\begin{subequations}
\begin{align}
\hbar \frac{\partial \xi_1}{\partial t_1} & = 
1 + t_1 \xi_{1 }^2 - 2 \xi_{1 }^2 \xi_{2 } 
+ \xi_{1 }^3 \xi_{3 } - 2 \xi_{1 }^4 \xi_{4 }, 
\label{eq:mirror-5/2-1}\\[+.3em]
\hbar \frac{\partial \xi_2}{\partial t_1} & = 
2 t_1 \xi_{1 } \xi_{2 } - 4 \xi_{1 } \xi_{2 }^2 
+ 2 \xi_{1 }^2 \xi_{2 } \xi_{3 } 
- 4 \xi_{1 }^3 \xi_{2 } \xi_{4 }, 
\\[+.3em]
\hbar \frac{\partial \xi_3}{\partial t_1} & = 
4 t_1 \xi_{2 } - 8 \xi_{2 }^2 
+ 4 \xi_{1 } \xi_{2 } \xi_{3 } 
- 8 \xi_{1 }^2 \xi_{2 } \xi_{4 }
-\frac{4 t_2 \xi_{1 }^2}{\xi_{2 }}, 
\\[+.3em]
\hbar \frac{\partial \xi_4}{\partial t_1} & =  
\frac{t_1 \xi_{3 }}{2} + \frac{\xi_{1 } \xi_{3 }^2}{2} - \xi_{2 } \xi_{3 } 
- 2 t_1 \xi_{1 } \xi_{4 } + 4 \xi_{1 } \xi_{2 } \xi_{4 } 
- 3 \xi_{1 }^2 \xi_{3 } \xi_{4 } + 4 \xi_{1 }^3 \xi_{4 }^2 
-\frac{2 t_2 \xi_{1 }}{\xi_{2 }}.
\label{eq:mirror-5/2-4}
\end{align}
\end{subequations}
The Laurent series solution \eqref{eq:LaurenGar5/2q1}--\eqref{eq:LaurenGar5/2p2}
corresponds to the solution of the system \eqref{eq:mirror-5/2-1}--\eqref{eq:mirror-5/2-4} 
at $t_1 = \alpha$ with the initial value
\begin{equation} \label{eq:inicond}
(\xi_1(\alpha), \xi_2(\alpha), \xi_3(\alpha), \xi_4(\alpha)) = 
\left(0, \beta, - 4 \gamma - \hbar, 
\frac{\alpha^2 + 14 \alpha \beta - 32 \beta^2 + 45 \delta}{18} \right).
\end{equation}
Since we have assumed $\beta \ne 0$, we can apply the Cauchy existence theorem, 
which implies that the solution of the system \eqref{eq:mirror-5/2-1}--\eqref{eq:mirror-5/2-4}  
with the initial condition \eqref{eq:inicond} is holomorphic around $t_1 = \alpha$. 
It further guarantees the convergence of the Laurent series 
\eqref{eq:LaurenGar5/2q1}--\eqref{eq:LaurenGar5/2p2}. 
This completes the proof.
\end{proof}

\subsubsection{Existence of singularity reduction $SR_{\rm Gar_{5/2+3/2}}$}

Recall that $Q_{\rm Gar_{5/2+3/2}}$ and $A_{{\rm Gar_{5/2+3/2}},j}$
are defined in \eqref{eq:Sch-pot} and \eqref{eq:GarAj}
obtained from the ismonodormy system associated with ${\rm Gar}_{5/2+3/2}$. 
The following claim on the existence of the 
singularity reduction is an analogue of Theorem \ref{thm:SR-Gar9/2}.

\begin{thm}
\label{thm:SR-Gar5/2} 
The functions $Q_{\rm Gar_{5/2+3/2}}$ and $A_{{\rm Gar_{5/2+3/2}},2}$,
with $(q_1,q_2,p_1,p_2)$ substituted by the Laurent series solution 
\eqref{eq:LaurenGar5/2q1}--\eqref{eq:LaurenGar5/2p2}
of the Hamiltonian system \eqref{eq:GarHam}, 
have finite limits as $t_1 \to \alpha$: 
\begin{align}
\lim_{t_1 \to \alpha} Q_{\rm Gar_{5/2+3/2}}(x,t_1, t_2)  
& = R_{\rm Gar_{5/2+3/2}}(x,t_2) \notag \\
& = R_{\rm Gar_{5/2+3/2},0}(x,t_2) 
+ \hbar R_{{\rm Gar_{5/2+3/2}},1}(x,t_2) 
+ \hbar^2 R_{{\rm Gar_{5/2+3/2}},2}(x,t_2),
\notag \\[+.3em]
& 
\label{eq:lim-QGar5/2}
\\
\lim_{t_1 \to \alpha} A_{{\rm Gar_{5/2+3/2}},2}(x,t_1, t_2) 
&  = B_{\rm Gar_{5/2+3/2}}(x,t_2) = \frac{\beta x}{t_2(x- \beta)}, 
\label{eq:BGar5/2}
\end{align}
where 
\begin{subequations}
\begin{align}
R_{\rm Gar_{5/2+3/2},0} & = 
x - \alpha + \frac{\alpha^2 + 32 \alpha \beta - 50 \beta^2 + 45 \delta}{18 x} 
\notag \\[+.5em]
& \quad 
- \frac{18 t_2 + \beta (\alpha^2 \beta + 14 \alpha \beta^2 
- 32 \beta^3 - 18 \gamma^2 + 45 \beta \delta)}{18 \beta x^2}
+ \frac{t_2}{x^3}, 
& \quad 
\label{eq:RG5/2-0}\\[+.5em]
R_{\rm Gar_{5/2+3/2},1} & = \frac{(3 x - \beta) \gamma}{2 x^2 (x - \beta)},
\label{eq:RG5/2-1} \\[+.5em]
R_{\rm Gar_{5/2+3/2},2} & = \frac{5 x^2 + 10 x \beta - 3 \beta^2}{16 x^2 (x - \beta)^2}.
\label{eq:RG5/2-2}
\end{align}
\end{subequations}
\end{thm}
\begin{proof}
It can be checked by 
a straightforward computation 
using \eqref{eq:LaurenGar5/2q1}--\eqref{eq:LaurenGar5/2p2}. 
\end{proof}

Theorem \ref{thm:SR-Gar5/2} implies that, 
when taking the limit $t_1 \to \alpha$, 
the singularity reduction
\begin{equation} \label{eq:SR-Gar5/2}
SR_{\rm Gar_{5/2+3/2}}~:~ 
\left( \hbar^2 \frac{\partial^2}{\partial x^2} - R_{\rm Gar_{5/2+3/2}}(x, t_2)  \right)  \psi = 0
\end{equation}
of the isomonodromy system associated with ${\rm Gar}_{5/2+3/2}$ exists.
As well as the case of ${\rm Gar_{9/2}}$, 
for generic choice of the free parameters $(\alpha,\beta,\gamma,\delta)$,
the classical limit 
\begin{equation}
y^2 = R_{\rm Gar_{5/2+3/2},0}(x, t_2)
\label{eq:CL-Gar5/2}
\end{equation}
of $SR_{\rm Gar_{5/2+3/2}}$ defines a family of hyper-elliptic curves of genus 2 parametrized by $t_2$.

\subsubsection{Isomonodromy property of $SR_{\rm Gar_{5/2+3/2}}$  and quasi-Painlev\'e property}

It follows from \eqref{eq:RG5/2-0}--\eqref{eq:RG5/2-2} that 
the the regular singular point $x = \beta$ of
the Schr\"odinger-type equation $SR_{\rm Gar_{5/2+3/2}}$  
is an apparent singular point. 
Furthermore, we can show the following result 
which is similar to Theorem \ref{thm:second-IM-9/2}.

\begin{thm} 
\label{thm:second-IM-5/2}
If the parameters $(\alpha,\beta,\gamma,\delta)$ are functions of $t_2$ satisfying 
the system of ODEs
\begin{subequations}
\begin{align}
\hbar \frac{d\alpha}{dt_2} & = -\frac{\beta\hbar}{t_2},
\label{eq:SP-5/2-1}
\\[+.5em]
\hbar \frac{d\beta}{dt_2} & = -\frac{\beta (4 \gamma + \hbar)}{2 t_2},
\\[+.5em]
\hbar \frac{d\gamma}{dt_2} & = \frac{18 t_2 - \alpha^2 \beta^2 + 4 \alpha \beta^3 
- 4 \beta^4 - 45 \beta^2 \delta}{18 t_2 \beta},
\\[+.5em]
\hbar \frac{d\delta}{dt_2} & = \frac{2(32 \alpha \beta \gamma 
- 100 \beta^2 \gamma + 9 \alpha \beta \hbar - 18 \beta^2 \hbar)}{45t_2},
\label{eq:SP-5/2-4}
\end{align}
\end{subequations}
then $R = R_{\rm Gar_{5/2+3/2}}$ satisfies the equation \eqref{eq:comp-SR-9/2}
with the choice $B = B_{\rm Gar_{5/2+3/2}}$ given in \eqref{eq:BGar5/2}. 
Therefore, $SR_{\rm Gar_{5/2+3/2}}$  
with $(\alpha, \beta, \gamma, \delta)$ 
satisfying \eqref{eq:SP-5/2-1}--\eqref{eq:SP-5/2-4} 
is compatible with the PDE 
\begin{equation} \label{eq:Lax-SR-5/2}
\frac{\partial\psi}{\partial t_2}=
    \left( B_{\rm Gar_{5/2+3/2}} \frac{\partial}{\partial x} - 
    \frac{1}{2} \frac{\partial B_{\rm Gar_{5/2+3/2}}}{\partial x} \right) \psi,
\end{equation}
and thus $SR_{\rm Gar_{5/2+3/2}}$  forms an isomonodromic family of linear ODEs parametrized by $t_2$.
\end{thm}

This statement can also be proved by a straightforward calculation.
The system \eqref{eq:SP-5/2-1}--\eqref{eq:SP-5/2-4} can be found 
using the same approach as stated in Remark \ref{rem:finding-secondary-isomonodromy}.
Therefore, $SR_{\rm Gar_{5/2+3/2}}$ can be isomonodromically deformed 
with respect to $t_2$. 

Similar to the case of ${\rm Gar}_{9/2}$, 
we can derive a fourth-order single ODE 
\begin{align} 
& \hbar^2 \left( \alpha^{(4)} 
+ \frac{2 \alpha^{(3)}}{t_2} 
- \frac{4 \alpha'' \alpha^{(3)}}{\alpha'} 
-\frac{3 (\alpha'')^2}{t_2 \alpha'} 
+ \frac{3 (\alpha'')^3}{(\alpha')^2} 
\right) 
\notag \\[+.3em]
& \quad 
+ \frac{4 \alpha''}{t_2^2 \alpha'} 
+ 12 t_2 (\alpha')^3 \alpha''
+ 4 \alpha \, (\alpha')^2 \alpha'' 
+ \frac{4 \alpha\, (\alpha')^3}{t_2} 
+ 14 (\alpha')^4 
+ \frac{2}{t_2^3} 
\label{eq:4th-ODE-5/2}
\end{align}
from the system \eqref{eq:SP-5/2-1}--\eqref{eq:SP-5/2-4} of ODEs, 
which describes the isomonodromy deformation of $SR_{\rm Gar_{5/2+3/2}}$. 
As well as \eqref{eq:4th-ODE-9/2}, the equation \eqref{eq:4th-ODE-5/2} 
admits a Puiseux series solution of the following form:
\begin{align}
\alpha(t_2) 
& = 
c_1 - \frac{3^{1/3} \hbar^{2/3} (t_2 - b)^{1/3}}{b^{1/3}}
- \frac{c_1 (t_2 - b)}{5 b} 
+ \frac{5 \hbar^{2/3} (t_2 - b)^{4/3}}{4 \cdot 3^{2/3} b^{4/3}} 
+ \frac{c_2 (t_2 - b)^{5/3}}{3^{1/3} b^{5/3} \hbar^{2/3}} 
\notag \\[+.3em]
& \quad 
+ \frac{3 c_1 (t_2 - b)^2}{10 b^2}
+ \frac{c_3 (t_2 - b)^{7/3}}{3^{2/3} b^{7/3} \hbar^{4/3}} 
+ \frac{(81 c_1^2 - 3275 c_2) (t_2 - b)^{8/3}}{1500 \cdot 3^{1/3} b^{8/3} \hbar^{2/3}} 
\notag \\[+.3em]
& \quad 
+ \frac{(3375 b - 27 c_1^4 - 1575 c_1^2 c_2 - 12500 c_2^2 + 450 c_1 (7 c_3 - 5 \hbar^2)) 
(t_2 - b)^3}{12375 b^3 \hbar^2}
\notag \\[+.3em]
& \quad 
+ O((t_2-b)^{{10}/{3}}), 
\label{eq:PSer-5/2}
\end{align}
where $(b, c_1, c_2, c_3)$ can be taken as free parameters.
Consequently, we conclude that the equation \eqref{eq:4th-ODE-5/2} 
describes the isomonodromy deformation of $SR_{\rm Gar_{5/2+3/2}}$, 
but it does not possess the Painlev\'e property.

Furthermore, as well as the case of ${\rm Gar_{9/2}}$, 
$R_{\rm Gar_{5/2+3/2}}$ with $\alpha$ 
substituted by the Puiseux series \eqref{eq:PSer-5/2}
also has a finite limit as $t_2 \to b$:
\begin{align}
& \lim_{t_2 \to b} R_{\rm Gar_{5/2+3/2}}(x,t_2) \notag \\[+.3em] 
& \quad = 
\left( x - c_1 + \frac{180 c_1^2 + 1750 c_2}{450 x} 
-\frac{6 c_1^3 + 175 c_1 c_2 - 175 c_3}{75 x^2}
+ \frac{b}{x^3} \right) + \frac{5\hbar^2}{2x^2}.
\end{align}
This means that $SR_{\rm Gar_{5/2+3/2}}$ also admits 
the second singularity reduction as $t_2 \to b$, 
and the classical limit of the reduced Schr\"odinger-type ODE 
defines a hyper-elliptic curve of genus 2 if we choose generic $(b, c_1, c_2, c_3)$.

In this way, 
we have obtained results parallel to those of Dubrovin--Kapaev \cite{DK14}
for ${\rm Gar}_{5/2+3/2}$. 
As mentioned in Section \ref{sec:intro}, 
it would be interesting to apply a similar analysis to other Garnier systems 
and fourth-order Painlev\'e equations listed in \cite{Kaw17, KNS18}.

\begin{rem}
Similarly to the case of ${\rm Gar}_{9/2}$, there are several possible pole orders 
in the Laurent series solutions of $\rm Gar_{5/2+3/2}$. 
In addition to \eqref{eq:LaurenGar5/2q1}--\eqref{eq:LaurenGar5/2p2}, 
we found another Laurent series solution of ${\rm Gar}_{5/2+3/2}$ of the form 
\begin{subequations}
\begin{align}
q_1(t_1) & = 
-\frac{\hbar}{t_1 - \alpha} 
- \frac{(2 t_2 - \alpha \beta)(t_1 - \alpha)}{3 \beta \hbar} 
- \frac{(4 \beta \gamma - 3 \hbar)(t_1 - \alpha)^2}{4 \hbar^2} 
+ \frac{\delta (t_1 - \alpha)^3}{\hbar^3}
\notag \\[+.1em]
& \quad + O((t_1-\alpha)^4),
\label{eq:LaurenGar5/2q1alt}\\[+.3em]
q_2(t_1) & = 
\frac{\beta (t_1 - \alpha)^2}{\hbar^2} + \frac{(2 t_2 - \alpha \beta)(t_1 - \alpha)^4}{3 \hbar^4} + \frac{\beta (4 \beta \gamma - 3 \hbar)(t_1 - \alpha)^5}{6 \hbar^5}
+ O((t_1-\alpha)^6),
\\[+.3em]
p_1(t_1) & = 
\frac{\hbar^2}{(t_1 - \alpha)^2} 
+ \frac{t_2 + \alpha \beta}{3 \beta} 
+ \frac{t_1 - \alpha}{2} 
+ \frac{(4 t_2^2 - 4 t_2 \alpha \beta + \beta^2 \alpha^2 + 9 \beta^2 \delta)
(t_1 - \alpha)^2}{18 \beta^2 \hbar^2}
\notag \\[+.1em]
& \quad + O((t_1-\alpha)^3), 
\\[+.3em]
p_2(t_1) & 
= \frac{t_2 \hbar^3}{\beta^2 (t_1 - \alpha)^3} 
+ \frac{\gamma \hbar^2}{(t_1 - \alpha)^2} 
+ \frac{4 \alpha \gamma \beta^2 - 12 t_2 \beta \gamma + 3 t_2 \hbar}{12 \beta^2}
+ O((t_1-\alpha)^1).
\label{eq:LaurenGar5/2p2alt}
\end{align}
\end{subequations}
where we assume $\beta \ne 0$.
In contrast to the case observed in Remark \ref{rem:another-Laurent-9/2}, 
in this case, we obtain a Laurent series that contains four free parameters 
$(\alpha, \beta, \gamma, \delta)$.
We can also verify that the above Laurent series solution provides 
the singularity reduction of the isomonodromy system for ${\rm Gar_{5/2+3/2}}$ 
which is different from \eqref{eq:SR-Gar5/2}: 
\begin{equation} \label{eq:SR-Gar5/2-alt}
\widetilde{SR}_{\rm Gar_{5/2+3/2}}~:~ 
\left( \hbar^2 \frac{\partial^2}{\partial x^2} - \widetilde{R}_{\rm Gar_{5/2+3/2}}(x, t_2)  \right) \psi = 0,
\end{equation}
where 
\begin{align}
\widetilde{R}_{\rm Gar_{5/2+3/2}}(x,t_2) 
& = 
\lim_{t_1 \to \alpha} Q_{\rm Gar_{5/2+3/2}}
 \notag \\[+.3em]
\quad 
& = 
\biggl( x - \alpha 
- \frac{50 t_2^2 - 32 t_2 \alpha \beta - \beta^2 \alpha^2 
+ 45 \beta^2 \delta}{18 \beta^2 x} 
\notag \\[+.5em]
& \quad + \frac{32 t_2^3 - 14 t_2^2 \alpha \beta - 18 \beta^4 
+ 18 \beta^5 \gamma^2 - t_2 \beta^2 \alpha^2 + 45 t_2 \beta^2 \delta}{18 \beta^3 x^2}
+ \frac{t_2}{x^3} \biggr)
\notag \\[+.5em] 
& \quad + \hbar \frac{t_2 \beta \gamma}{x^2(\beta x - t_2)}
- \hbar^2 \frac{\beta(\beta x - 4t_2)}{4x(\beta x - t_2)^2}.  
\label{eq:alt-SR-5/2}
\end{align}
The point $x = t_2/\beta$ is an apparent singular point of 
the singularity reduction \eqref{eq:SR-Gar5/2-alt}. 
Furthermore, in this case, by imposing that the 
Laurent series solution \eqref{eq:LaurenGar5/2q1alt}--\eqref{eq:LaurenGar5/2p2alt} 
also satisfy the Hamiltonian system with respect to $t_2$, 
we can derive the following system of differential equations that 
the free parameters must satisfy:
\begin{subequations}
\begin{align}
\hbar \frac{d\alpha}{dt_2} & = -\frac{\hbar}{\beta},
\label{eq:SP-5/2-1-alt}
\\[+.5em]
\hbar \frac{d\beta}{dt_2} & = \frac{2\beta^2 \gamma}{t_2},
\\[+.5em]
\hbar \frac{d\gamma}{dt_2} & = - \frac{4 t_2^3 - 4 t_2^2 \alpha \beta 
+ t_2 \alpha^2 \beta^2 - 18 \beta^4 
+ 36 \beta^5 \gamma^2 - 45 t_2 \beta^2 \delta}{18 t_2 \beta^4},
\\[+.5em]
\hbar \frac{d\delta}{dt_2} & = 
\frac{200 t_2 \beta \gamma - 114 t_2 \hbar - 64 \alpha \beta^2 \gamma 
+ 30 \alpha \beta \hbar}{45 \beta^2}.
\label{eq:SP-5/2-4-alt}
\end{align}
\end{subequations}
However, we found that the Laurent series solution 
\eqref{eq:LaurenGar5/2q1alt}--\eqref{eq:LaurenGar5/2p2alt}
provides the trivial limit of $A_{{\rm Gar_{5/2+3/2}},2}$ when $t_1 \to \alpha$:
\begin{equation}
\lim_{t_1 \to \alpha} A_{{\rm Gar_{5/2+3/2}},2}(x,t_1, t_2) 
= \widetilde{B}_{\rm Gar_{5/2+3/2}}(x,t_2) = 0,
\end{equation}
and consequently, the compatibility equation \eqref{eq:comp-SR-9/2} 
is not satisfied by the pair 
$(R, B) = (\widetilde{R}_{\rm Gar_{5/2+3/2}}, \widetilde{B}_{\rm Gar_{5/2+3/2}})$.
Thus, although the situation differs from that observed in Remark \ref{rem:another-Laurent-9/2} 
for $\rm Gar_{9/2}$, we have also found that the Laurent series solutions 
\eqref{eq:LaurenGar5/2q1alt}--\eqref{eq:LaurenGar5/2p2alt} 
of ${\rm Gar}_{5/2+3/2}$ does not give an analogue of Theorem \ref{thm:second-IM-5/2}.
As related to the point mentioned at the end of the Remark \ref{rem:another-Laurent-9/2},
our observation suggests that, 
for an analogue of Theorem \ref{thm:second-IM-5/2} to hold, 
additional conditions are required beyond the condition that 
the Laurent series of the Hamiltonian system with respect to the first variable 
is consistent with the Hamiltonian system with respect to the second variable.
\end{rem}




\end{document}